\documentclass[a4paper, 11pt]{article}
\usepackage{amsmath,amssymb,mathrsfs,amsthm}
\usepackage{array,enumerate}
\usepackage{latexsym}
\usepackage{mathrsfs}
\usepackage{amscd}
\usepackage[all]{xy}
\usepackage{graphicx}
\usepackage[svgnames]{xcolor}
\usepackage{tikz}
\usetikzlibrary{intersections}
\theoremstyle{definition}
\newtheorem{thm}{Theorem}[section]
\newtheorem{dfn}[thm]{Definition}
\newtheorem{note-dfn}[thm]{\rm{Notation-Definition}}
\newtheorem{note-rem}[thm]{\rm{Notation-Remark}}

\newtheorem{exam}[thm]{\rm{Example}}

\newtheorem{prop}[thm]{Proposition}
\newtheorem{cor}[thm]{Corollary}
\newtheorem{lem}[thm]{Lemma}
\newtheorem{rem}[thm]{\rm{Remark}}
\newtheorem{conj}[thm]{\rm{Conjecture}}
\newtheorem{ques}[thm]{\rm{Question}}

\newtheorem{prop-dfn}[thm]{Proposition-Definition}

\newcommand{\Z}{\mathbb{Z}}
\newcommand{\Q}{\mathbb{Q}}

\newcommand{\ZZ}{\widehat{\mathbb{Z}}}

\newcommand{\cS}{\mathcal{S}}
\newcommand{\Spec}{\mathrm{Spec}\,}

\title{Topics in the Grothendieck conjecture for hyperbolic polycurves of dimension $2$}
\author{Ippei Nagamachi}
\date{}

\begin{document}
\maketitle

\begin{abstract}
In this paper, we study the anabelian geometry of hyperbolic polycurves of dimension $2$ over sub-$p$-adic fields.
In $1$-dimensional case, Mochizuki proved the Hom version of the Grothendieck conjecture for hyperbolic curves over sub-$p$-adic fields and the pro-$p$ version of this conjecture. 
In $2$-dimensional case, a naive analogue of this conjecture does not hold for hyperbolic polycurves over general sub-$p$-adic fields.
Moreover, the Isom version of the pro-$p$ Grothendieck conjecture does not hold in general.
We explain these two phenomena and prove the Hom version of the Grothendieck conjecture for hyperbolic polycurves of dimension $2$ under the assumption that the Grothendieck section conjecture holds for some hyperbolic curves.
\end{abstract}

\tableofcontents

\setcounter{section}{-1}
\section{Introduction}
Let $K$ be a field, $\overline{K}$ a separable closure of $K$, and $Y, X$ normal varieties (cf.\,Definition \ref{geometrically}) over $K$.
Write $Y_{\overline{K}}$ (resp.\,$X_{\overline{K}}$) for the scheme $Y\times_{\Spec K}\Spec \overline{K}$ (resp.\,$X\times_{\Spec K}\Spec \overline{K}$) and $G_{K}$ for the absolute Galois group $\mathrm{Gal}\,(\overline{K}/K)$.
Take a geometric point $\ast_{Y}$ (resp.\,$\ast_{X}$) of $Y_{\overline{K}}$ (resp.\,$X_{\overline{K}}$).
A morphism $f: Y \rightarrow X$ over $K$ induces a homomorphism
$$f_{\ast}: \pi_{1}(Y,\ast_{Y})\rightarrow \pi_{1}(X, \ast_{X})$$
over $G_{K}$ between the \'etale fundamental groups of $Y$ and $X$ which is uniquely determined up to inner automorphisms induced by elements of $\pi_{1}(X_{\overline{K}}, \ast_{X})$.
Hence, we obtain a natural map
$$\mathrm{Mor}_{K}(Y, X) \rightarrow \mathrm{Hom}_{G_{K}}(\pi_{1}(Y,\ast_{Y}), \pi_{1}(X, \ast_{X}))/\mathrm{Inn}\,\pi_{1}(X_{\overline{K}}, \ast_{X}),$$
where we write $\mathrm{Mor}_{K}(Y, X)$ (resp.\,$\mathrm{Hom}_{G_{K}}(\pi_{1}(Y,\ast_{Y}), \pi_{1}(X, \ast_{X}))$; $\mathrm{Inn}\,\pi_{1}(X_{\overline{K}}, \ast_{X})$) for the set of morphisms from $Y$ to $X$ over $K$ (resp.\,the set of continuous homomorphisms over $G_{K}$ from $\pi_{1}(Y,\ast_{Y})$ to $\pi_{1}(X, \ast_{X})$; the group of inner automorphisms of $\pi_{1}(X_{\overline{K}}, \ast_{X})$).

In anabelian geometry, the following questions have been studied:
\begin{ques}
\begin{enumerate}
\item Write $\mathrm{Isom}_{K}(Y, X)$ (resp.\,$\mathrm{Isom}_{G_{K}}(\pi_{1}(Y,\ast_{Y}), \pi_{1}(X, \ast_{X}))$) for the subset of $\mathrm{Mor}_{K}(Y, X)$ (resp.\,$\mathrm{Hom}_{G_{K}}(\pi_{1}(Y,\ast_{Y}), \pi_{1}(X, \ast_{X}))$) consisting of isomorphisms.
Is the map
$$\mathrm{Isom}_{K}(Y, X) \rightarrow \mathrm{Isom}_{G_{K}}(\pi_{1}(Y,\ast_{Y}), \pi_{1}(X, \ast_{X}))/\mathrm{Inn}\,\pi_{1}(X_{\overline{K}}, \ast_{X})$$
bijective?
\item Write $\mathrm{Mor}^{\mathrm{dom}}_{K}(Y, X)$ for the subset of $\mathrm{Mor}_{K}(Y, X)$ consisting of dominant morphisms and $\mathrm{Hom}^{\mathrm{open}}_{G_{K}}(\pi_{1}(Y,\ast_{Y}), \pi_{1}(X, \ast_{X}))$ for the subset of $\mathrm{Hom}_{G_{K}}(\pi_{1}(Y,\ast_{Y}), \pi_{1}(X, \ast_{X}))$ consisting of open homomorphisms.
Is the map (cf.\,\cite{Ho} Lemma 1.3)
$$\mathrm{Mor}^{\mathrm{dom}}_{K}(Y, X) \rightarrow \mathrm{Hom}^{\mathrm{open}}_{G_{K}}(\pi_{1}(Y,\ast_{Y}), \pi_{1}(X, \ast_{X}))/\mathrm{Inn}\,\pi_{1}(X_{\overline{K}}, \ast_{X})$$
bijective?
\item Suppose that $Y=\Spec K$.
(Hence, we have $\mathrm{Mor}_{K}(Y, X)=X(K)$).
Write $\mathrm{Sect}_{G_{K}}(\pi_{1}(X, \ast_{X}))$ for the set of sections of the natural surjective homomorphism $\pi_{1}(X, \ast_{X})\rightarrow G_{K}$.
Is the map
$$X(K) \to \mathrm{Sect}_{G_{K}}(\pi_{1}(X, \ast_{X}))/\mathrm{Inn}\,\pi_{1}(X_{\overline{K}}, \ast_{X})$$
bijective?
\end{enumerate}
\label{question}
\end{ques}

In the case where $K$ is finitely generated over $\Q$ and $X$ is a hyperbolic curve (cf.\,Definition \ref{hyperbolic}.1), Grothendieck conjectured that the maps discussed in Questions \ref{question}.1, \ref{question}.2, and a modified version of the map discussed in Question 0.1.3 (see Conjecture \ref{section conjecture} for this modified version) are bijective \cite{Letter}.
Question \ref{question}.1 (resp.\,\ref{question}.2; \ref{question}.3) is called the Isom version of the Grothendieck conjecture (resp.\,the Hom version of the Grothendieck conjecture; the Grothendieck section conjecture).

Suppose that $X$ is a hyperbolic curve.
In the case where $K$ is finitely generated over $\Q$, $Y$ is also a hyperbolic curve, and at least one of $X$ and $Y$ is affine, Question \ref{question}.1 was affirmatively answered by Tamagawa \cite{Tama1}.
In the case where $K$ is a sub-$p$-adic field (i.e., a subfield of a field finitely generated over $\Q_{p}$ (cf.\,Definition \ref{subp})) and $Y$ is a smooth variety, Question \ref{question}.2 was affirmatively answered by  Mochizuki (cf.\,\cite{Moch} Theorem A).
Also, the injectivity portion of Question \ref{question}.3 was proved in \cite{Moch} (cf.\,Lemma \ref{well-def}).

Suppose that $X$ is a hyperbolic polycurve (cf.\,Definition \ref{hyperbolic}), that is, a variety $X$ over $K$ which admits a structure of successive smooth fibrations
\begin{equation}\label{0}
X = X_n \overset{f_n}{\rightarrow} X_{n-1} \overset{f_{n-1}}{\rightarrow} \cdots 
\overset{f_2}{\rightarrow} X_1 
\overset{f_1}{\rightarrow} \mathrm{Spec}\,K 
\end{equation}
whose fibers are hyperbolic curves.
A hyperbolic polycurve is regarded as a higher dimensional analogue of a hyperbolic curve, and has been studied in anabelian geometry.
In the case where $K$ is sub-$p$-adic and $n \leq 4$, Question \ref{question}.2 was affirmatively answered by Hoshi under some conditions (cf.\,\cite{Ho} Theorem A).
Then he solved Question \ref{question}.1 as a corollary.
Moreover, in the case where $X$ is a strongly hyperbolic Artin neighborhood (\cite{SS} Definition 6.1) and $K$ is finitely generated over $\Q$, Question \ref{question}.1 was affirmatively answered by Stix and Schmidt \cite{SS}.

Suppose that $X$ is a hyperbolic polycurve of dimension $2$.
\cite{Ho}\,Theorem 3.14, which is a sort of the Hom version of the Grothendieck conjecture, states that every element of the set
$$\mathrm{Hom}^{\mathrm{open}}_{G_{K}}(\pi_{1}(Y,\ast_{Y}), \pi_{1}(X, \ast_{X}))/\mathrm{Inn}\,\pi_{1}(X_{\overline{K}}, \ast_{X})$$
with topologically finitely generated kernel arises from an element of the set $\mathrm{Mor}^{\mathrm{dom}}_{K}(Y, X)$.
(See \cite{Ho4} Theorem B for a generalization of this theorem.)
On the other hand, since there exists a $K$-morphism $f: Y\to X$ which is not dominant and induces an open outer homomorphism between the \'etale fundamental groups, we cannot expect that Question \ref{question}.2 is affirmative (cf.\,\cite{SGA2} XII Corollaire 3.5).
However, we can expect that any open outer group homomorphism from $\pi_{1}(Y, \ast_{Y})$ to $\pi_{1}(X, \ast_{X})$ over $G_{K}$ arises from a nonconstant $K$-morphism from $Y$ to $X$.

One of the main results of this paper is as follows:
\begin{thm}[cf.\,Theorem \ref{assumesection}]
Suppose that $K$ is a sub-$p$-adic field and $Y$ is a normal variety over $K$.
Let $X_{2}\to X_{1}\to \Spec K$ be a hyperbolic polycurve of dimension $2$ over $K$ (cf.\,Definition \ref{hyperbolic}.2) and suppose that $X=X_{2}$.
Moreover, suppose that the Grothendieck section conjecture (cf.\,Question \ref{question}.3 and Conjecture \ref{section conjecture}) holds for every hyperbolic curve over a field which is finitely generated extension of $K$ with transcendental degree $1$ (cf.\,Remark \ref{trdeg1}).
Then each element of
$$\mathrm{Hom}_{G_{K}}^{\mathrm{open}}(\pi_{1}(Y, \ast_{Y}), \pi_{1}(X_{2}, \ast_{X}))/ \mathrm{Inn}\,\pi_{1}(X_{2,\overline{K}}, \ast_{X})$$
arises from an element of $\mathrm{Mor}^{\mathrm{nonconst}}_{K}(Y, X_{2})$.
Here, $\mathrm{Mor}^{\mathrm{nonconst}}_{K}(Y, X_{2})$ denotes the subset of $\mathrm{Mor}_{K}(Y, X_{2})$ consisting of nonconstant morphisms.
\label{introsectionhom}
\end{thm}

In \cite{Moch}, the Isom and Hom versions of the pro-$p$ Grothendieck conjecture for hyperbolic curves over sub-$p$-adic fields were studied.
Sawada studied the Isom and Hom versions of the pro-$p$ Grothendieck conjecture for hyperbolic polycurves over sub-$p$-adic fields under some conditions on their fundamental groups \cite{Saw}.
In Section \ref{example}, we give examples of hyperbolic polycurves over sub-$p$-adic fields which show that the Isom and Hom versions of  the pro-$p$ Grothendieck conjecture for hyperbolic polycurves over sub-$p$-adic fields do not hold generally.

The content of each section is as follows:

In Section \ref{convention}, we give a review of properties of the \'etale fundamental groups of hyperbolic polycurves.
In Section \ref{conjecturestate}, we review the Grothendieck section conjecture for hyperbolic curves over sub-$p$-adic fields.
In Section \ref{mainsection}, we give a proof of Theorem \ref{introsectionhom}.
In Section \ref{example}, we give examples of hyperbolic polycurves which show that the anabelianity of hyperbolic polycurves is weaker than that of hyperbolic curves in some sense.

Acknowledgements: The author thanks Yuichiro Hoshi for various useful comments, and especially for the following: (i) informing me of the arguments used in Theorem \ref{assumesection}; (ii) explaining to me various results about the Grothendieck section conjecture.
This work was supported by the Research Institute for Mathematical Sciences, an International Joint Usage/Research Center located in Kyoto University.\\

\noindent\textbf{Terminologies for outer homomorphisms of groups:}
Let $G_{1}$ and $G_{2}$ be profinite groups.
An outer homomorphism $G_{1} \to G_{2}$ is defined to be an equivalence class of continuous homomorphisms $G_{1} \to G_{2}$, where two such homomorphisms are considered equivalent if they differ by composition with an inner automorphism of $G_{2}$.
Let $\phi:G_{1}\to G_{2}$ be an outer group homomorphism.
Note that the kernel of $\phi$ is uniquely determined and the image of $\phi$ is determined uniquely up to conjugation.
We shall say that $\phi$ is open (or, alternatively, $\phi$ is an outer open homomorphism) if the image of $\phi$ is open.

\label{Intro}

\section{Notation and basic properties of the \'etale fundamental groups of hyperbolic curves}
In this section, we fix some notations and definitions.
We also prove some properties of inertia subgroups of the \'etale fundamental groups of hyperbolic curves (cf.\,Proposition \ref{decompchange}).

We start with the definition of hyperbolic curves.

\begin{dfn}
Let $S$ be a scheme.
\begin{enumerate}
\item We shall say that a scheme $X$ is a hyperbolic curve over $S$ if the following conditions are satisfied:
\begin{itemize}
\item $X$ is a scheme over $S$.
\item There exists a scheme $\overline{X}$ proper smooth over $S$ with connected $1$-dimensional geometric fibers of genus $g$.
\item There exists an effective Cartier divisor $D$ of $\overline{X}$ which is finite \'etale over $S$ of rank $r$.
\item The open subscheme $\overline{X}\setminus D$ of $\overline{X}$ is isomorphic to $X$ over $S$.
\item $2g+r-2>0$.
\end{itemize}
\item We shall say that $X_{2} \rightarrow X_{1} \rightarrow S$ is a hyperbolic polycurve of relative dimension $2$ over $S$ if $X_{2} \rightarrow X_{1}$ and $X_{1} \rightarrow S$ are hyperbolic curves.
\end{enumerate}
\label{hyperbolic}
\end{dfn}

\begin{rem}
Let $S$ be a normal scheme and $X$ a hyperbolic curve over $S$.
Then a pair of schemes $(\overline{X}, D)$ which satisfies the conditions in Definition \ref{hyperbolic}.1 is uniquely determined by $X$ up to canonical isomorphism from the argument given in the discussion entitled “Curves” in \cite{Moch4} $\S 0$.
We shall refer to $D$ as the divisor of cusps of the hyperbolic curve $X\rightarrow S$.
\label{cptunique}
\end{rem}

\begin{dfn}
Let $K$ be a field.
We shall say that a scheme $X$ over $K$ is a variety if the morphism $X \rightarrow \Spec K$ is separated and of finite type with geometrically connected fibers.
\label{geometrically}
\end{dfn}

\begin{dfn}
Let $p$ be a prime number.
We shall say that a field $K$ is a sub-$p$-adic field if there exist a finitely generated extension field $L$ over $\Q_{p}$ and an injective homomorphism from $K$ to $L$.
\label{subp}
\end{dfn}

\begin{prop}
Let $S$ be a connected locally Noetherian separated normal scheme over $\Q$ and $X \rightarrow S$ a hyperbolic curve.
Write $D$ for the divisor of cusps of $X\rightarrow S$.
\begin{enumerate}
\item
The divisor $D$ is a disjoint union of finitely many normal schemes which are \'etale over $S$.
\item
Let $D_{0}$ be an irreducible component of $D$.
Take a geometric point $\ast$ of $X$.
Choose a decomposition group $G_{d}$ of $D_{0}$ in $\pi_{1}(X, \ast)$ and write $\overline{G_{d}}$ for the image of $G_{d}$ in $\pi_{1}(S,\ast)$.
Then we have the following natural commutative diagram of profinite groups with exact horizontal lines and injective vertical arrows:
\begin{equation}
\xymatrix{
1\ar[r]&\ZZ(1)\ar[r] \ar@{_{(}->}[d] 
&G_{d}\ar@{_{(}->}[d]\ar[r]&\overline{G_{d}}\ar[r]\ar@{_{(}->}[d]&1\\
1\ar[r]&\Delta_{X/S} \ar[r] 
&\pi_{1}(X, \ast)\ar[r]&\pi_{1}(S, \ast)\ar[r]&1.
}
\end{equation}
Here, we write $\Delta_{X/S}$ for the kernel of the homomorphism $\pi_{1}(X, \ast) \to \pi_{1}(S, \ast)$.
Moreover, $\overline{G_{d}}$ is isomorphic to the \'etale fundamental group of $D_{0}$ in a canonical way up to inner automorphism of $\pi_{1}(S,\ast)$.
\item
Let $S'$ be another connected locally Noetherian separated normal scheme and $S' \rightarrow S$ a dominant morphism.
Suppose that $\ast\to X$ factors through $\ast\to X\times_{S}S'\to X$.
Write $D'_{0}$ for the irreducible component of the divisor of cusps of $X\times_{S}S'\rightarrow S'$ over $D_{0}$ determined by $G_{d}$ and $G'_{d}$ for the decomposition group of $D'_{0}$ in $\pi_{1}(X\times_{S}S', \ast)$ over $G_{d}$.
Then we have a natural isomorphism $G'_{d}\cong G_{d}\times_{\pi_{1}(S,\ast)}\pi_{1}(S',\ast)$.
\end{enumerate}
\label{decompchange}
\end{prop}

\begin{proof}
Since the morphism $D \rightarrow S$ is \'etale, the assertion 1 holds.
Next, we show the assertion 2.
We may assume that $\ast$ is a geometric generic point.
Let $K(S)$ be the function field of $S$ and $G_{K(S)}$ the absolute Galois group of $K(S)$ determined by $\ast$.
Write $X_{K(S)}$ for the scheme $X\times_{S}\Spec K(S)$. 
Then $D_{0}\times_{S}\Spec K(S)$ is an irreducible component of the divisor of cusps of the hyperbolic curve $X_{K(S)} \rightarrow \Spec K(S)$.
Choose a decomposition group $G^{K(S)}_{d}$ of $D_{0}\times_{S}\Spec K(S)$ in $\pi_{1}(X_{K(S)}, \ast)$ over $G_{d}$ and write $\overline{G_{d}}^{K(S)}$ for the image of $G_{d}^{K(S)}$ in $G_{K(S)}$.
We obtain the following diagram of profinite groups with exact horizontal lines by \cite{Tama1} Lemma (2.2) and \cite{Ho} Proposition 2.4 (i)(ii):
\begin{equation*}
\xymatrix{
1\ar[r]&\ZZ(1)\ar[r] \ar@{_{(}->}[d] 
&G^{K(S)}_{d}\ar@{_{(}->}[d]\ar[r]&\overline{G_{d}}^{K(S)}\ar[r]\ar@{_{(}->}[d]&1\\
1\ar[r]&\Delta_{X/S} \ar[r]  \ar@{=}[d] 
&\pi_{1}(X_{K(S)}, \ast)\ar@{->>}[d]\ar[r]&G_{K(S)}\ar[r]\ar@{->>}[d]&1\\
1\ar[r]&\Delta_{X/S} \ar[r] 
&\pi_{1}(X, \ast)\ar[r]&\pi_{1}(S, \ast)\ar[r]&1.
}
\end{equation*}
Note that $D_{0}\times_{S}\Spec K(S)$ is the spectrum of a finite separable extension field of $K(S)$ and, by \cite{Tama1} Lemma (2.2), $\overline{G_{d}}^{K(S)}$ is isomorphic to the absolute Galois group of this field.
Since the homomorphism $G^{K(S)}_{d}\rightarrow G_{d}$ is surjective and $D_{0}$ is finite \'etale over $S$, the assertion 2 holds.
The assertion 3 follows from the assertion 2 and \cite{Ho} Proposition 2.4 (i)(ii).
\end{proof}

\label{convention}

\section{The Grothendieck section conjecture}
In this section, we recall the Grothendieck section conjecture for hyperbolic curves over sub-$p$-adic fields.

Let $K$ be a field of characteristic $0$, $\overline{K}$ an algebraic closure of $K$, $G_{K}$ the absolute Galois group $\mathrm{Gal}\,(\overline{K}/K)$, $X$ a hyperbolic curve over $K$, and $D$ the divisor of cusps of the hyperbolic curve $X$.
Write $X_{\overline{K}}$ for the scheme $X\times_{\Spec K}\Spec \overline{K}$.
Take a geometric point $\ast$ of $X_{\overline{K}}$.
Write $\mathrm{Sect}_{G_{K}}(\pi_{1}(X,\ast))$ for the set of continuous sections of the homomorphism $\pi_{1}(X,\ast)\rightarrow G_{K}$.

First, we state ``the Grothendieck section conjecture'' in a general setting.
\begin{conj}[cf.\,Question \ref{question}.3]
\begin{enumerate}
\item Suppose that $X$ is a proper hyperbolic curve over $K$.
Then the natural map
\begin{equation}
X(K) \to \mathrm{Sect}_{G_{K}}(\pi_{1}(X, \ast))/\mathrm{Inn}\,\pi_{1}(X_{\overline{K}}, \ast)
\label{sectionconj}
\end{equation}
is bijective.
\item Write $\mathrm{Sect}_{G_{K}}^{CD}(\pi_{1}(X,\ast))$ for the subset of $\mathrm{Sect}_{G_{K}}(\pi_{1}(X,\ast))$ consisting of sections whose images are contained in a decomposition group of some closed point of $D$.
Then the map (\ref{sectionconj}) induces a map
\begin{equation}
X(K)\rightarrow (\mathrm{Sect}_{G_{K}}(\pi_{1}(X,\ast))\setminus \mathrm{Sect}_{G_{K}}^{CD}(\pi_{1}(X,\ast)))/\mathrm{Inn}\,\pi_{1}(X_{\overline{K}},\ast)
\label{section}
\end{equation}
and this map is bijective (cf.\,Example \ref{decompsection}.1).
\end{enumerate}
\label{section conjecture}
\end{conj}

\begin{lem}
Suppose that $K$ is a sub-$p$-adic field.
The map (\ref{section}) is well-defined and injective.
\label{well-def}
\end{lem}

\begin{proof}
The well-definedness portion follows from the proof of \cite{Moch3} Theorem 1.3 (iv) and \cite{Moch} Theorem C.
The injectivity portion follows from \cite{Moch} Theorem C.
\end{proof}

\begin{rem}
\begin{enumerate}
\item
Suppose that $K$ is a generalized sub-$p$-adic (not necessarily sub-$p$-adic) field.
In this case, as written in \cite{Ho2} Introduction, the injectivity portion of Lemma \ref{well-def} also holds (cf.\,the proof of \cite{Moch} Theorem C and \cite{Moch2} Theorem 4.12 and Remark following Theorem 4.12).
Moreover, the well-definedness portion of Lemma \ref{well-def} also holds by its proof.
\item There exist (generalized) sub-$p$-adic fields such that the Grothendieck section conjecture does not hold for hyperbolic curves over them.
Let $p$ be a prime number and suppose that $K$ is the field of fractions of a henselization of $\Z_{(p)}$.
Write $\widehat{K}$ for the completion of the field $K$.
Let $\overline{\widehat{K}}$ be an algebraic closure of $\widehat{K}$ and fix an embedding $\overline{K}\hookrightarrow \overline{\widehat{K}}$ over $K$.
Then we have $\mathrm{Gal}(\overline{K}/K) \cong \mathrm{Gal}(\overline{\widehat{K}}/\widehat{K})$.
Suppose that $X(\widehat{K})$ has uncountably infinitely many $\widehat{K}$-rational points.
(For example, suppose that $X$ has a $K$-rational point $x$ and a finite morphism $X\to\mathbb{P}^{1}_{K}$ \'etale at $x$.
Then, by the theory of locally analytic manifolds and the implicit function theorem, $X$ has uncountably infinitely many $\widehat{K}$-rational points.)
Since the cardinality of the set $X(K)$ is at most countable, the induced map $X(K) \to X(\widehat{K})$ is not surjective.
Therefore, the Grothendieck section conjecture for $X$ does not hold.
\end{enumerate}
\label{henselexam}
\end{rem}

\begin{exam}
Suppose that $D$ has a $K$-rational point $x$.
\begin{enumerate}
\item
We show that, in the case where $X$ is affine, the map (\ref{sectionconj}) is not surjective in general.
The decomposition group of $x$ in the fundamental group $\pi_{1}(X, \ast)$ is isomorphic to the absolute Galois group $G_{K((T))}$ of the field of Laurent series over $K$ by \cite{Tama1} Lemma (2.2).
Since the characteristic of $K$ is $0$, there exists a continuous section of the homomorphism $G_{K((T))}\rightarrow G_{K}$.
(Indeed, we can construct such a section by considering a compatible system $(T^{1/n})_{n\geq 1}$.)
Therefore, we obtain a section $G_{K}\rightarrow \pi_{1}(X, \ast)$ which is not defined by a rational point of $X$ by Lemma \ref{well-def}.
\item
Here, we give an example of outer homomorphism over $G_{K}$ between the \'etale fundamental group of hyperbolic curves over $K$.
We do not fix geometric points and do not write base points of \'etale fundamental groups.
The morphism $\Spec K((T)) \rightarrow \Spec K[T,1/T]$ induces an outer isomorphism
$$(G_{K((T))}=\,)\,\pi_{1}(\Spec K((T)) ) \rightarrow \pi_{1}(\Spec K[T,\frac{1}{T}] )$$
between their fundamental groups.
By composing the surjective outer homomorphism $\pi_{1}(\mathbb{P}^{1}_{K}\setminus\{0,1,\infty\})\rightarrow \pi_{1}(\Spec K[T,\frac{1}{T}])$ induced by the open immersion $\mathbb{P}^{1}_{K}\setminus\{0,1,\infty\}\to \Spec K[T,\frac{1}{T}]$, the inverse of the above outer isomorphism, and an outer isomorphism from $G_{K((T))}$ to a decomposition group of $x$ in $\pi_{1}(X)$, we obtain an outer homomorphism $\phi: \pi_{1}(\mathbb{P}^{1}_{K}\setminus\{0,1,\infty\})\to\pi_{1}(X)$ whose image is a decomposition group of $x$.
Therefore, $\mathrm{Im}\,\phi$ neither is open in $\pi_{1}(X)$ nor determines a section of the homomorphism $\pi_{1}(X) \rightarrow G_{K}$.
\end{enumerate}
\label{decompsection}
\end{exam}

\label{conjecturestate}

\section{Sections for hyperbolic polycurves of dimension $2$ }
In this section, we prove the Hom version of the Grothendieck conjecture for morphisms from regular varieties to hyperbolic polycurves of dimension $2$ over sub-$p$-adic fields under the assumption that the Grothendieck section conjecture for hyperbolic curves holds.

Let $K$ be a field of characteristic $0$, $X_{2}\rightarrow X_{1}\rightarrow \Spec K$ a hyperbolic polycurve of dimension $2$ over $K$, $K_{1}$ the function field of $X_{1}$, $\overline{K}_{1}$ an algebraic closure of $K_{1}$, and $\overline{K}$ the algebraic closure of $K$ in $K_{1}$.
Write $G_{K}$ (resp.\,$G_{K_{1}}$) for the absolute Galois group $\mathrm{Gal}(\overline{K}_{1}/K_{1})$ (resp.\,$\mathrm{Gal}(\overline{K}/K)$) and $X_{2, K_{1}}$ for the scheme $X_{2}\times_{X_{1}}\Spec K_{1}$.
In this section, for any normal variety $W$ over $K$ or $K_{1}$, we consider a geometric point of $W\times_{\Spec K}\Spec \overline{K}$ or $W\times_{\Spec K_{1}}\Spec \overline{K}_{1}$ and write $\Pi_{W}$ (resp.\,$\Delta_{W}$) for the \'etale fundamental group of $W$ (resp.\,$W\times_{\Spec K}\Spec \overline{K}$ or $W\times_{\Spec K_{1}}\Spec \overline{K}_{1}$).
We omit base points of \'etale fundamental groups in this notation, because we only consider outer homomorphisms unless otherwise noted.
Write $\Delta_{2,1}$ for the kernel of the homomorphism $\Pi_{X_{2}}\to\Pi_{X_{1}}$ induced by the structure morphism  $X_{2} \to X_{1}$.
Since the profinite group $\Pi_{X_{2,K_{1}}}$ is isomorphic to the profinite group $\Pi_{X_{2}}\times_{\Pi_{X_{1}}}G_{K_{1}}$ by \cite{Ho} Proposition 2.4 (ii), we have the following commutative diagram of profinite groups with exact horizontal lines:
\begin{align*}
\xymatrix{
1\ar[r]&\Delta_{2,1}\ar@{=}[d]\ar[r]&\Pi_{X_{2,K_{1}}}\ar[d]\ar[r]&G_{K_{1}}\ar[d]\ar[r]&1\\
1\ar[r]&\Delta_{2,1}\ar[r]&\Pi_{X_{2}}\ar[r]&\Pi_{X_{1}}\ar[r]&1.
}
\end{align*}
We write $\mathrm{Sect}_{\Pi_{X_{1}}}(\Pi_{X_{2}})$ for the set of continuous sections of the homomorphism $\Pi_{X_{2}}\to \Pi_{X_{1}}$.
Let $(\overline{X}_{2}, D)$ be the smooth compactification of the hyperbolic curve $X_{2} \rightarrow X_{1}$ (cf.\,Remark \ref{cptunique}).
Since $X_{1}$ is normal, we have a decomposition $D=\underset{1\leq i \leq n}{\amalg} D_{i}$ by Proposition \ref{decompchange}.1, where each $D_{i}$ is a normal scheme.
Write $\theta_{i}$ for the generic point of $D_{i}$.
We shall write $\mathrm{Sect}^{CD}_{\Pi_{X_{1}}}(\Pi_{X_{2}})$ for the set of continuous sections of the homomorphism $\Pi_{X_{2}}\rightarrow \Pi_{X_{1}}$ whose images are contained in a decomposition group of some $\theta_{i}$ in $\Pi_{X_{2}}$.

\begin{lem} There exists a natural injective map
\begin{equation*}
\mathrm{Sect}_{\Pi_{X_{1}}}(\Pi_{X_{2}})/\mathrm{Inn}\,(\Delta_{2,1})
\rightarrow \mathrm{Sect}_{G_{K_{1}}}(\Pi_{X_{2,K_{1}}})/\mathrm{Inn}\,(\Delta_{2,1})
\end{equation*}
which induces a map
\begin{equation*}
\mathrm{Sect}^{CD}_{\Pi_{X_{1}}}(\Pi_{X_{2}})/\mathrm{Inn}\,(\Delta_{2,1})
\rightarrow \mathrm{Sect}^{CD}_{G_{K_{1}}}(\Pi_{X_{2,K_{1}}})/\mathrm{Inn}\,(\Delta_{2,1})
\end{equation*}
and a map
\begin{equation*}
\begin{split}
&(\mathrm{Sect}_{\Pi_{X_{1}}}(\Pi_{X_{2}})\setminus\mathrm{Sect}^{CD}_{\Pi_{X_{1}}}(\Pi_{X_{2}}))/\mathrm{Inn}\,(\Delta_{2,1})\\
\rightarrow& (\mathrm{Sect}_{G_{K_{1}}}(\Pi_{X_{2,K_{1}}})\setminus\mathrm{Sect}^{CD}_{G_{K_{1}}}(\Pi_{X_{2,K_{1}}}))/\mathrm{Inn}\,(\Delta_{2,1}).
\end{split}
\end{equation*}
\label{naturalmap}
\end{lem}

\begin{proof}
Since the group $\Pi_{X_{2,K_{1}}}$ is isomorphic to the group $\Pi_{X_{2}}\times_{\Pi_{X_{1}}}G_{K_{1}}$ by \cite{Ho} Proposition 2.4 (ii), we obtain a natural map
$$\mathrm{Sect}_{\Pi_{X_{1}}}(\Pi_{X_{2}})/\mathrm{Inn}\,(\Delta_{2,1})
\rightarrow \mathrm{Sect}_{G_{K_{1}}}(\Pi_{X_{2,K_{1}}})/\mathrm{Inn}\,(\Delta_{2,1}).$$
The injectivity of this map follows from the surjectivity of the homomorphism $G_{K_{1}}\rightarrow \Pi_{X_{1}}$.

Let $s_{X}: \Pi_{X_{1}} \to \Pi_{X_{2}}$ be a section of the homomorphism $\Pi_{X_{2}}\to \Pi_{X_{1}}$ and $\theta$ an element of $\{\theta_{i}\mid 1\leq i \leq n \}$.
Write $G_{d}^{XK}$ for a decomposition group of $\theta$ in $\Pi_{X_{2},K_{1}}$ and $G_{d}^{X}$ for the image of $G_{d}^{XK}$ in $\Pi_{X_{2}}$.
Note that $G_{d}^{X}$ coincides with a decomposition group of $\theta$ in $\Pi_{X_{2}}$.
Write $s_{XK}$ for the section of the homomorphism $\Pi_{X_{2,K_{1}}}\to G_{K_{1}}$ determined by $s_{X}$.
It suffices to show that the image of the homomorphism $s_{XK}$ is contained in $G_{d}^{XK}$ if and only if the image of the homomorphism $s_{X}$ is contained in $G_{d}^{X}$.
This follows from Proposition \ref{decompchange}.3.
Hence, we finish the proof of Lemma \ref{naturalmap}.
\end{proof}

\begin{thm}
Let $X_{2}\to X_{1} \to \Spec K$ be a hyperbolic polycurve of dimension $2$  over $K$.
Suppose that the Grothendieck section conjecture holds for the hyperbolic curve $X_{2,K_{1}}\to \Spec K_{1}$.
Then the natural map
\begin{equation}
\mathrm{Sect}_{X_{1}}(X_{2}) \to
(\mathrm{Sect}_{\Pi_{X_{1}}}(\Pi_{X_{2}}))/\mathrm{Inn}\,(\Delta_{2,1})
\label{curvedarrow}
\end{equation}
factors through
\begin{equation}
\mathrm{Sect}_{X_{1}}(X_{2}) \to (\mathrm{Sect}_{\Pi_{X_{1}}}(\Pi_{X_{2}})\setminus\mathrm{Sect}^{CD}_{\Pi_{X_{1}}}(\Pi_{X_{2}}))/\mathrm{Inn}\,(\Delta_{2,1})
\label{isomarrow}
\end{equation}
and the homomorphism (\ref{isomarrow}) is bijective.
\label{fromsection}
\end{thm}

\begin{proof}
Consider the following diagram:
\begin{equation*}
\xymatrix{
\mathrm{Sect}_{X_{1}}(X_{2}) \ar@{.>}[r] \ar@{_{(}->}[d]\ar@/^13pt/[rr]
&\cS(\Pi\setminus CD)\ar@{_{(}->}[d]\ar@{^{(}->}[r]
&\cS(\Pi)\ar@{_{(}->}[d]\\
X_{2, K_{1}}(K_{1})\ar[r] 
&\cS(G\setminus CD)\ar@{^{(}->}[r]
&\cS(G),
}
\end{equation*}
where we write $\cS(\Pi\setminus CD)$ (resp.\,$\cS(\Pi)$; $\cS(G\setminus CD)$; $\cS(G)$) for the set $(\mathrm{Sect}_{\Pi_{X_{1}}}(\Pi_{X_{2}})\setminus\mathrm{Sect}^{CD}_{\Pi_{X_{1}}}(\Pi_{X_{2}}))/\mathrm{Inn}\,(\Delta_{2,1})$ (resp.\,$(\mathrm{Sect}_{\Pi_{X_{1}}}(\Pi_{X_{2}}))/\mathrm{Inn}\,(\Delta_{2,1})$; $(\mathrm{Sect}_{G_{K_{1}}}(\Pi_{X_{2,K_{1}}})\setminus\mathrm{Sect}^{CD}_{G_{K_{1}}}(\Pi_{X_{2,K_{1}}}))/\mathrm{Inn}\,(\Delta_{2,1})$; $(\mathrm{Sect}_{G_{K_{1}}}(\Pi_{X_{2,K_{1}}}))/\mathrm{Inn}\,(\Delta_{2,1})$).
The right rectangle is discussed in Lemma \ref{naturalmap}.
The first vertical arrow is induced by base change, and hence injective.
The curved arrow in the first horizontal line is (\ref{curvedarrow}) and the biggest rectangle of the diagram is commutative.
The left homomorphism of the second horizontal line is bijective by the assumption of Theorem \ref{fromsection}.
By using these discussion and Lemma \ref{naturalmap}, (\ref{isomarrow}) is induced and injective.
Moreover, each element of
$$(\mathrm{Sect}_{\Pi_{X_{1}}}(\Pi_{X_{2}})\setminus\mathrm{Sect}^{CD}_{\Pi_{X_{1}}}(\Pi_{X_{2}}))/\mathrm{Inn}\,(\Delta_{2,1})$$
is defined by a section of the morphism $X_{2}\rightarrow X_{1}$ by \cite{Ho} Lemma 2.10 and the surjectivity of the first homomorphism of the second horizontal line.
Hence, we finish the proof of Theorem \ref{fromsection}.
\end{proof}

\begin{cor}
\label{propercase}
Suppose that the morphism $X_{2}\to X_{1}$ is proper and the Grothendieck section conjecture holds for the hyperbolic curve $X_{2,K_{1}}\to \Spec K_{1}$.
Then the map $\mathrm{Sect}_{X_{1}}(X_{2}) \to \mathrm{Sect}_{\Pi_{X_{1}}}(\Pi_{X_{2}})/\mathrm{Inn}\,(\Delta_{2,1})$ is bijective.
\end{cor}

\begin{proof}
Since the morphism $X_{2}\to X_{1}$ is proper, we have $\mathrm{Sect}^{CD}_{\Pi_{X_{1}}}(\Pi_{X_{2}})=\emptyset$.
Therefore, Corollary \ref{propercase} follows from Theorem \ref{fromsection}.
\end{proof}

\begin{thm}
Suppose that $K$ is a sub-$p$-adic field.
Let $Y$ be a normal variety over $K$.
Suppose that the Grothendieck section conjecture holds for every hyperbolic curve over a field which is finitely generated over $K$ of transcendental degree $1$ (cf.\,Remark \ref{trdeg1}).
Then, for any outer open homomorphism $\phi\in\mathrm{Hom}_{G_{K}}^{\mathrm{open}}(\Pi_{Y}, \Pi_{X_{2}})/\mathrm{Inn}(\Delta_{X})$, there exists a nonconstant morphism $Y \rightarrow X$ inducing $\phi$.
\label{assumesection}
\end{thm}

\begin{proof}
Write $\phi_{1}$ for the composite outer homomorphism
$$\Pi_{Y}\overset{\phi}{\to} \Pi_{X_{2}}\to  \Pi_{X_{1}}.$$
Then the outer homomorphism $\phi_{1}$ is induced by a unique dominant $K$-morphism $f_{1}: Y \to X_{1}$ by \cite{Ho}\,Theorem 3.3.
Write $K'_{1}$ for the normalization of $K_{1}$ in the function field of $Y$, $\eta$ for the scheme $\Spec\,K'_{1}$, $X'_{1}$ for the open subscheme of the normalization of $X_{1}$ in $K'_{1}$ determined by the image of $Y$, $Y_{\eta}$ for the scheme $Y\times_{X'_{1}}\eta$, and $G_{K'_{1}}$ for the \'etale fundamental group of $\eta$ (, which is isomorphic to the absolute Galois group of $K'_{1}$).
Then we have the following commutative diagram of profinite groups:
$$
\xymatrix{
\mathrm{Ker}\,(\Pi_{Y_{\eta}}\to G_{K'_{1}})\ar[r]\ar@{.>}[dd]&\Pi_{Y_{\eta}}\ar@{.>}[rd]\ar[d]\ar[rr]&&G_{K'_{1}}\ar@{=}[d]\\
&\Pi_{Y}\ar[rd]_{\phi}&\Pi_{X_{2}}\times_{\Pi_{X_{1}}}G_{K'_{1}}\ar[r]\ar[d]&G_{K'_{1}}\ar[d]\\
\Delta_{2,1}\ar[rr]&&\Pi_{X_{2}}\ar[r]&\Pi_{X_{1}}.
}
$$
If the image of the induced outer homomorphism
\begin{equation}
\mathrm{Ker}\,(\Pi_{Y_{\eta}}\to G_{K'_{1}})\to \Delta_{2,1}
\label{repair}
\end{equation}
is nontrivial, $\phi$ arises from a morphism $Y\to X_{2}$ over $K$ by \cite{Ho} Lemma 3.4 (iv).
Suppose that the outer homomorphism (\ref{repair}) is trivial.
Note that we have natural isomorphisms $\Pi_{X_{2}\times_{X_{1}}\eta}\cong \Pi_{X_{2}}\times_{\Pi_{X_{1}}}G_{K'_{1}}$ and $\mathrm{Ker}(\Pi_{X_{2}\times_{X_{1}}\eta}\to G_{K'_{1}})\simeq\Delta_{2,1}$ by [2] Proposition 2.4 (ii) and the outer homomorphism $\Pi_{Y_{\eta}}\to G_{K'_{1}}$ is surjective.
Hence, the image of the induced outer homomorphism $\Pi_{Y_{\eta}}\to \Pi_{X_{2}\times_{X_{1}}\eta}$ defines a section $s$ of the outer homomorphism $\Pi_{X_{2}\times_{X_{1}}\eta} \to G_{K'_{1}}$.
Suppose that
$$s\in \mathrm{Sect}^{CD}_{G_{K'_{1}}}(\Pi_{X_{2}\times_{X_{1}}\eta})/\mathrm{Inn}(\Delta_{2,1}).$$
Then the group $\mathrm{Im}(\Pi_{Y_{\eta}}\to \Pi_{X_{2}})$ is not open in $\Pi_{X_{2}}$ by Proposition \ref{decompchange}.2 and 3.
Since the outer homomorphism $\Pi_{Y_{\eta}}\to \Pi_{Y}$ is surjective, the image of $\phi$ coincides with $\mathrm{Im}(\Pi_{Y_{\eta}}\to \Pi_{X_{2}})$.
Therefore, the image of $\phi$ is not open, which contradicts the assumption on $\phi$.
By the Grothendieck section conjecture for the hyperbolic curve $X_{2}\times_{X_{1}}\eta \to \eta$, we have a $K'_{1}$-morphism $Y_{\eta}\to X_{2}\times_{X_{1}}\eta$ inducing the outer homomorphism $\Pi_{Y_{\eta}} \to  \Pi_{X_{2}\times_{X_{1}}\eta}$.
Then by \cite{Ho} Lemma 2.10, there exists a $K$-morphism $Y\to X_{2}$ inducing $\phi$ such that the composite morphism $Y_{\eta}\to Y\to X_{2}$ coincides with the composite morphism $Y_{\eta}\to X_{2}\times_{X_{1}}\eta \to X_{2}$.
\end{proof}

\begin{rem}
Let $Y$ be as in Theorem \ref{assumesection}.
Suppose that the Grothendieck section conjecture holds for every hyperbolic curve over a field which is finitely generated over $K$ of transcendental degree $\mathrm{dim}\,Y$ and the morphism $X_{2}\to X_{1}$ is proper.
Write $\eta$ (resp.\,$G_{\eta}$) for the spectrum (resp.\,the absolute Galois group) of the function field of $Y$.
Then we have a diagram of profinite groups
$$
\xymatrix{
G_{\eta} \ar[d] \ar@{.>}[dr]_{\phi_{\eta}} \ar@{=}[drr] & & \\
\Pi_{Y} \ar[rd]_{\phi}
&\Pi_{X_{2}\times_{X_{1}}\eta} \ar[r] \ar[d]
&G_{\eta}\ar[d]\\
&\Pi_{X_{2}} \ar[r] 
&\Pi_{X_{1}},
}
$$
where $\phi_{\eta}$ is the outer homomorphism induced by using the isomorphism
$$\Pi_{X_{2}\times_{X_{1}}\eta} \cong \Pi_{X_{2}}\times_{\Pi_{X_{1}}}G_{\eta}.$$
By Grothendieck section conjecture and \cite{Ho} Lemma 2.10, we can prove that $\phi$ is induced by a $K$-morphism $Y \to X_{2}$.
Then we can show Theorem \ref{assumesection} without using the assumption that $\phi$ is open.
\label{trdeg1}
\end{rem}

\label{mainsection}

\section{Examples of hyperbolic polycurves}
In this section, we give examples of hyperbolic polycurves which show that the anabelianity of hyperbolic polycurves is weaker than that of hyperbolic curves in some sense.

As we write in Section \ref{Intro}, Mochizuki proved the Hom version of the pro-$p$ Grothendieck conjecture for hyperbolic curves over sub-$p$-adic fields (cf.\,\cite{Moch}).
Moreover, Sawada proved a pro-$p$ analogue of \cite{Ho} Theorem A under a certain assumption on the \'etale fundamental groups of hyperbolic polycurves (cf.\,\cite{Saw}).
We construct examples which show that the Isom version of the pro-$p$ Grothendieck conjecture for hyperbolic polucurves over sub-$p$-adic fields does not hold in general in this section.

Let $K$ be a field of characteristic $0$, $\overline{K}$ an algebraic closure of $K$, and $p$ a prime number.

\begin{note-dfn}
\begin{enumerate}
\item Let $G$ be a profinite group.
We write $G^{p}$ for the maximal pro-$p$ quotient of $G$ (i.e., the inverse limit of the inverse system consisting of the quotient groups of $G$ by open normal subgroups such that the orders of the quotient groups are powers of $p$).
\item For any variety $X$ over $K$, we write $\Pi_{X}$ (resp.\,$\Delta_{X}$; $\Pi_{X}^{(p)}$) for the \'etale fundamental group of $X$ (resp.\,the \'etale fundamental group of $X\times_{\Spec K}\Spec \overline{K}$; the quotient group $\Pi_{X}/\mathrm{Ker} (\Delta_{X}\to \Delta_{X}^{p})$) in this section.
\end{enumerate}
\end{note-dfn}

First, we prove an elementary lemma.

\begin{lem}
Let
$$1 \rightarrow N \rightarrow G \rightarrow H \rightarrow 1$$
be an exact sequence of profinite groups.
\begin{enumerate}
\item We have an exact sequence
$$(N/[N,\mathrm{Ker}(G \rightarrow G^{p})])^{p} \rightarrow G^{p} \rightarrow H^{p} \rightarrow 1.$$
Here, ``$[-,-]$" denotes the topological closure of the commutator subgroup.
\item Suppose that we have a section $s: H \rightarrow G$ of the homomorphism $G \rightarrow H$ and write $N_{\mathrm{Ker}(H \rightarrow H^{p})}$ for the maximal quotient group of $N$ on which $\mathrm{Ker}(H \rightarrow H^{p})$ acts trivially.
Then we have an exact sequence
$$(N_{\mathrm{Ker}(H \rightarrow H^{p})})^{p} \rightarrow G^{p} \rightarrow H^{p} \rightarrow 1.$$
\end{enumerate}
\label{pro-l-exact}
\end{lem}

\begin{proof}
\begin{enumerate}
\item
Since the image of $[N,\mathrm{Ker}(G \rightarrow G^{p})]$ in $G^{p}$ is trivial, we obtain an exact sequence
$$N/[N,\mathrm{Ker}(G \rightarrow G^{p})] \rightarrow G^{p} \rightarrow H^{p} \rightarrow 1$$
and hence also an exact sequence
$$(N/[N,\mathrm{Ker}(G \rightarrow G^{p})])^{p} \rightarrow G^{p} \rightarrow H^{p} \rightarrow 1.$$
\item Since we have $s(\mathrm{Ker}(H \rightarrow H^{p})) \subset \mathrm{Ker}(G \rightarrow G^{p})$, the assertion follows from 1.
\end{enumerate}
\end{proof}

We show a lemma for Example \ref{not-inj}.

\begin{lem}
Suppose that $p\neq 2$.
Let $H$ be a hyperelliptic curve over $K$ and $\iota$ the hyperelliptic involution of $H$.
Suppose that there exist $K$-rational points $h$, $h'$ of $H$ which are fixed by the action of $\iota$.
By considering a geometric point over the fixed point $h$, we obtain actions of $\iota$ on $\Delta_{H\setminus\{ h' \}}$ and $\Delta_{H}$. 
Then we have
$(\Delta_{H})_{\langle \iota \rangle}^{p} = \{ 1 \}$ and $(\Delta_{H\setminus\{ h' \}})_{\langle \iota \rangle}^{p} = \{ 1 \}.$
\label{hyperelliptic}
\end{lem}

\begin{proof}
Since the profinite groups $\Delta_{H\setminus\{ h' \}}$ and $\Delta_{H}$ are topologically finitely generated, it suffices to show that $(\Delta_{H})_{\langle \iota \rangle}^{p,\mathrm{ab}} = \{ 1 \}$ and $(\Delta_{H\setminus\{ h' \}})_{\langle \iota \rangle}^{p,\mathrm{ab}} = \{ 1 \}.$
By \cite{Ho3} Lemma 1.11, the action of $\iota$ on the abelian profinite group $(\Delta_{H\setminus\{ h' \}}^{p,\mathrm{ab}}) \cong (\Delta_{H}^{p,\mathrm{ab}})$ is same as the multiplication by $-1$.
Therefore,
$$(\Delta_{H}^{p,\mathrm{ab}})_{\langle \iota \rangle} = \Delta_{H}^{p,\mathrm{ab}}/2\Delta_{H}^{p,\mathrm{ab}} = \{ 1 \}.$$
\end{proof}

\begin{exam}
Suppose that $p\neq 2$ and $K$ is a finite extension field of $\Q_{p}$.
We construct a proper hyperbolic polycurve $Z$ over a field $K$, such that the natural map
$$\mathrm{Isom}_{K}(Z, Z) \rightarrow \mathrm{Isom}_{G_{K}}(\Pi_{Z}^{(p)}, \Pi_{Z}^{(p)})/\mathrm{Inn}(\Delta_{Z}^{p})$$
is not injective.
Here, $\mathrm{Isom}_{K}(Z, Z)$ is the set of automorphisms of $Z$ over $K$, and $\mathrm{Isom}_{G_{K}}(\Pi_{Z}^{(p)}, \Pi_{Z}^{(p)})$ is the set of automorphisms of $\Pi_{Z}^{(p)}$ over $G_{K}$.
This shows that it is impossible to detect an automorphism of a hyperbolic polycurve from the corresponding $G_{K}$-outer automorphism of its pro-$p$ fundamental group.
In particular, the Isom version of the pro-$p$ Grothendieck conjecture, which is true for hyperbolic curves (\cite{Moch}) or hyperbolic polycurves with suitable conditions up to dimension $4$ (\cite{Saw}), cannot be true for general hyperbolic polycurves.

Let $X_{1}$ be a proper hyperbolic curve over $K$, and assume that there exists a homomorphism $\Pi_{X_{1}} \rightarrow \Z/2\Z$ which induces a surjection $ \Delta_{X_{1}} \rightarrow \Z/2\Z$.
We write $X_{1}' \rightarrow X_{1}$ for the \'etale covering space of $X_{1}$ corresponding to $\mathrm{Ker}\,(\Pi_{X_{1}} \rightarrow \Z/2\Z)$ and $\iota_{1}$ for a generator of $\mathrm{Aut}(X_{1}'/X_{1})$. 
Let $X_{2}$ be a hyperbolic curve over $K$ whose automorphism group over $K$ has a subgroup isomorphic to $\Z/2\Z= \langle\iota_{2}\rangle$ such that $X_{2}$ has a fixed point $x_{2}$ under the action of $\Z/2\Z( = \langle\iota_{2}\rangle)$. Moreover, assume that the maximal quotient group $(\Delta_{X_{2}})^{p}_{\Z/2\Z}$ of $(\Delta_{X_{2}})^{p}$ on which $\Z/2\Z$ acts trivially via a geometric point over $x_{2}$ is trivial (cf.\,Lemma \ref{hyperelliptic}).

Consider the action of $\Z/2\Z$ on $X_{2} \times_{\mathrm{Spec}\,K}X_{1}'$ induced by $(\iota_{2},\iota_{1})$.
Write $Z$ for the quotient scheme of $X_{2} \times_{\mathrm{Spec}\,K}X_{1}'$ by this $\Z/2\Z$-action.
By construction, we have a Cartesian diagram
\[
\xymatrix{ 
X_{2} \times_{\mathrm{Spec}\,K}X_{1}' \ar[r] \ar[d]
& X_{1}' \ar[d]\\
Z \ar[r] 
& X_{1}.
}
\]
Since the morphism $X'_{1}\to X_{1}$ is finite etale, $Z \to X_{1}$ is a hyperbolic curve whose geometric generic fiber coincides with that of $X_{2}\times_{\Spec K}X'_{1} \to X'_{1}$.
Hence, we obtain exact sequences of profinite groups
\begin{equation*}
1 \rightarrow \Delta_{X_{2}} \rightarrow \Pi_{Z} \rightarrow \Pi_{X_{1}} \rightarrow1
\end{equation*}
and
\begin{equation*}
1 \rightarrow \Delta_{X_{2}} \rightarrow \Delta_{Z} \rightarrow \Delta_{X_{1}}  \rightarrow1
\end{equation*}
by \cite{Ho} Proposition 2.4 (i).
Since the section $X_{1}' \to X_{2}\times_{\mathrm{Spec}\,K}X_{1}'$ of the morphism $X_{2}\times_{\mathrm{Spec}\,K}X_{1}'\to X_{1}'$ determined by the point $x_{2}$ is compatible with the actions of $\Z/2\Z$, we have a section $X_{1} \to Z$ of the morphism $Z \to X_{1}$ by taking the quotient schemes by $\Z/2\Z$.
Therefore, the homomorphism $\Pi_{Z} \to \Pi_{X_{1}}$ has a section which also determines a section of the homomorphism $\Delta_{Z} \to \Delta_{X_{1}}$.
We calculate the action
\begin{equation}
\Delta_{X_{1}} \rightarrow \mathrm{Aut}(\Delta_{X_{2}})
\label{goal}
\end{equation}
induced by the section.
Write $\psi$ for the composite homomorphism
\begin{equation*}
\begin{split}
\Delta_{X_{1}} &\rightarrow \Delta_{X_{1}}/\Delta_{X_{1}'}\simeq\Pi_{X_{1}}/\Pi_{X_{1}'}\\
& = \langle \iota_{1} \rangle \cong \Z/2\Z \\
& \cong \langle \iota_{2} \rangle \subset \{f \in \mathrm{Aut}(X_{2}/\mathrm{Spec}\,K)\mid f(x_{2})=x_{2} \} \rightarrow \mathrm{Aut}(\Delta_{X_{2}}).
\end{split}
\end{equation*}
By the construction of $Z$, the action (\ref{goal}) coincides with $\psi$.

Since the image of the composite homomorphism
$$\mathrm{Ker}(\Delta_{X_{1}}\rightarrow\Delta_{X_{1}}^{p}) \subset \Delta_{X_{1}} \subset \Pi_{X_{1}} \overset{\phi + \psi}{\to} \mathrm{Aut}(\Delta_{X_{2}})$$
is $\langle \iota_{2} \rangle$ by the assumption $2 \neq p$, the group $\mathrm{Ker}\,(\Delta_{Z}^{p}\to \Delta_{X_{1}}^{p})$ is a quotient group of $(\Delta_{X_{2}})_{\langle \iota_{2} \rangle}^{p}$ by Lemma \ref{pro-l-exact}.2.
Thus, we have
$$\Delta_{Z}^{p} \cong \Delta_{X_{1}}^{p}$$
by the assumption that $(\Delta_{X_{2}})^{p}_{\langle \iota_{2} \rangle}$ is trivial.
Hence, we have
$$\Pi_{Z}^{(p)} \cong \Pi_{X_{1}}^{(p)}.$$

It suffices to show that the scheme $Z$ has a nontrivial automorphism over $X$, since such an automorphism induces the trivial outer automorphism of $\Pi_{Z}^{(p)}(\cong\Pi_{X}^{(p)})$ (over $G_{K}$).
Since the automorphism $(\iota_{2}, \mathrm{id}_{X_{1}'})$ of $X_{2} \times_{\mathrm{Spec}\,K} X_{1}'$ over $X'_{1}$ is compatible with the diagonal action of $\Z/2\Z$, this automorphism defines a nontrivial automorphism of $Z$ over $X_{1}$.

Even if we change $X_{2}$ to another hyperbolic curve satisfying the above condition for $X_{2}$, the geometrically pro-$p$ \'etale fundamental group ($\Pi^{(p)} = \Pi/\mathrm{Ker}(\Delta \rightarrow \Delta^{p})$) of the resulting polycurve is isomorphic to $\Pi_{Z}^{(p)}$ over $G_{K}$, since we have the isomorphism $\Pi_{Z}^{(p)} \cong \Pi_{X_{1}}^{(p)}$.
Therefore this example gives a counterexample to the Isom version of the pro-$p$ Grothendieck conjecture for hyperbolic polycurves.
Since we have the isomorphism $\Delta_{Z}^{p} \cong \Delta_{X_{1}}^{p}$, we cannot even determine the dimension of a hyperbolic polycurve $X$ over $\overline{K}$ from its pro-$p$ \'etale fundamental group $\Delta_{X}^{p}$.
\label{not-inj}
\end{exam}

\begin{exam}
We give another example of non-isomorphic hyperbolic polycurves over a mixed characteristic local field $K$ with residual field of characteristic $p$ and of order $q$, whose geometrically pro-$p$ \'etale fundamental groups are isomorphic over $G_{K}$. This gives another counterexample to the Isom version of the pro-$p$ Grothendieck conjecture for hyperbolic polycurves.

Let $l$ be a prime number such that $l | q-1$.
Let $X_{2}$ be the hyperbolic curve $\mathbb{P}^{1}_{K} \setminus (\{ \infty \} \cup \mu_{l})$ over $K$.
Fix a primitive $l$-th root of unity $\zeta \in \mu_{l}$.
Let $\iota : \mathbb{P}^{1}_{K} \rightarrow \mathbb{P}^{1}_{K}$ be the automorphism $z \mapsto z\zeta$.
The morphism $\iota$ induces a $\Z/l\Z$-action on $X_{2}$ over $K$ which fixes $0 \in X_{2}$.
Let $X_{1}$ be a hyperbolic curve over $K$, and assume that there exists a homomorphism $\Pi_{X_{1}} \rightarrow \Z/l\Z$
which induces a surjection $ \Delta_{X_{1}} \rightarrow \Z/l\Z$.
We can obtain a scheme $Z$ via the construction same as that in Example \ref{not-inj} by replacing $\Z/2\Z$ by $\Z/l\Z$.
Then the fixed point $0 \in X_{2}$ defines a section $X_{1} \rightarrow Z$, which determines sections $\Delta_{X_{1}} \rightarrow \Delta_{Z}$ and $\Pi_{X_{1}} \rightarrow \Pi_{Z}$.
Since $p \neq l$, we obtain an exact sequence
$$(\Delta_{X_{2}})^{p}_{\langle \iota \rangle} \rightarrow \Delta_{Z}^{p} \rightarrow \Delta_{X_{1}}^{p} \rightarrow 1$$
by using the same argument as that in Example \ref{not-inj}.
The group $(\Delta_{X_{2}})^{p,\mathrm{ab}}_{\langle \iota \rangle}$ is generated by $1$ element, which shows that the group $(\Delta_{X_{2}})^{p}_{\langle \iota \rangle}$ is an abelian group.
Therefore, the kernel of the homomorphism $\Delta_{Z}^{p} \rightarrow \Delta_{X_{1}}^{p}$ is a quotient group of $(\Delta_{X_{2}})^{\mathrm{ab}}_{\langle \iota \rangle}$.
Since we have $(\Delta_{X_{2}})^{\mathrm{ab}}_{\langle \iota \rangle} = (\Delta_{X_{2}})^{\mathrm{ab}}_{\Delta_{X_{1}}} = (\Delta_{X_{2}}/[\Delta_{X_{2}},\Delta_{X_{2}}])_{\Delta_{X_{1}}} = \Delta_{X_{2}}/[\Delta_{X_{2}},\Delta_{Z}]$, we obtain the commutative diagram with exact horizontal lines 
\[
\xymatrix{
1 \ar[r]
&\Delta_{X_{2}}   \ar[r] \ar[d]
&\Delta_{Z}  \ar[r] \ar[d]
&\Delta_{X_{1}} \ar[r] \ar[d]
& 1\\
1 \ar[r]
&(\Delta_{X_{2}})^{\mathrm{ab}}_{\langle \iota \rangle} \ar[r] 
&\Delta_{Z}/[\Delta_{X_{2}},\Delta_{Z}]  \ar[r] 
&\Delta_{X_{1}} \ar[r]
&1.
}
\]
The second line of this diagram also splits, and thus we have the decomposition
$$\Delta_{Z}/[\Delta_{X_{2}},\Delta_{Z}] = (\Delta_{X_{2}})^{\mathrm{ab}}_{\langle \iota \rangle} \times\Delta_{X_{1}},$$
and hence the decomposition $(\Delta_{Z}/[\Delta_{X_{2}},\Delta_{Z}])^{p}\cong(\Delta_{X_{2}})^{p,\mathrm{ab}}_{\langle \iota \rangle} \times\Delta_{X_{1}}^{p}$.
Since
$$(\Delta_{X_{2}})^{p}_{\langle \iota \rangle} \cong (\Delta_{X_{2}})^{p,\mathrm{ab}}_{\langle \iota \rangle},$$
we have the isomorphism $\Delta_{Z}^{p} \cong (\Delta_{Z}/[\Delta_{X_{2}},\Delta_{Z}])^{p}$, and therefore we obtain the decomposition $\Delta_{Z}^{p} = (\Delta_{X_{2}})^{p,\mathrm{ab}}_{\langle \iota \rangle} \times\Delta_{X_{1}}^{p}$.
Note that $\Delta_{X_{2}}^{\mathrm{ab}}$ is isomorphic to $\ZZ(1) \otimes_{\Z} (\underset{z \in \mu_{l}}{\oplus} \Z e_{z})$ as a $\Pi_{X_{1}}$-module.
This shows that $\Pi_{Z}^{(p)}(=\Pi_{Z}/\mathrm{Ker}(\Delta_{Z} \rightarrow \Delta_{Z}^{p}))$ is isomorphic to $\Z_{p}(1) \rtimes \Pi_{X_{1}}^{(p)}$, which is defined by the action
$$\Pi_{X_{1}}^{(p)} (=\Pi_{X_{1}}/\mathrm{Ker}(\Delta_{X_{1}} \rightarrow \Delta_{X_{1}}^{p})) \rightarrow G_{K} \rightarrow \mathrm{Aut}(\Z_{p}(1)).$$
Therefore, $\Pi_{Z}^{(p)}$ does not depend on $l$. Moreover, if we consider the \'etale covering space of $Z$ corresponding to $p^{n}\Z_{p}(1) \rtimes \Pi_{X_{1}}^{(p)} \subset \Z_{p}(1) \rtimes \Pi_{X_{1}}^{(p)}$, its geometrically pro-$p$ \'etale fundamental group is isomorphic to $\Z_{p}(1) \rtimes \Pi_{X_{1}}^{(p)}$ over $G_{K}$.
However, the Euler characteristic of the \'etale covering space is larger than that of $Z$ and therefore it is not isomorphic to $Z$.

Note that the order of the group $\mathrm{Aut}(\Pi_{Z}^{(p)})/\mathrm{Inn}(\Delta_{Z}^{p})$ is infinite since it contains $\mathbb{Z}_{p}^{\times}$.
Also, note that the group $\Delta_{Z}^{p}$ is not center-free.
\end{exam}

\label{example}

(Ippei Nagamachi) Research Institute for Mathematical Sciences, Kyoto University

E-mail address: nagachi@kurims.kyoto-u.ac.jp
\end{document}